\documentclass[a4paper]{article}
\usepackage{amsmath}
\usepackage{amsfonts}
\usepackage[a4paper]{geometry}
\usepackage{graphics}

\newtheorem{theorem}{Theorem}

\newtheorem{corollary}[theorem]{Corollary}

\newtheorem{definition}[theorem]{Definition}
\newtheorem{example}[theorem]{Example}

\newtheorem{lemma}[theorem]{Lemma}

\newtheorem{remark}[theorem]{Remark}

\newenvironment{proof}[1][Proof]{\noindent\textbf{#1.} }{\ \rule{0.5em}{0.5em}}
\input{tcilatex}
\begin{document}

\title{Absolutely continuous copulas obtained by regularization of the Frech%
\'{e}t--Hoeffding bounds}
\author{Oscar Bj\"{o}rnham$^{1}$, Niklas Br\"{a}nnstr\"{o}m$^{1}$ and Leif
Persson$^{1,2}$ \and  \\
$^{1}$Swedish Defence Research Institute, FOI,\\
SE-901 82 Ume\aa , Sweden.\\
$^{2}$Department of Mathematics and Mathematical Statistics,\\
Ume\aa\ University,\\
SE-901 87 Ume\aa , Sweden}
\maketitle

\begin{abstract}
We show that the lower and upper Frech\'{e}t-Hoeffding copulas, which are
singular, can be regularized to absolutely continuous copulas. The method,
which is constructive and explicit, states sufficient conditions for when an
absolutely continuous copula can be achieved by averaging. A higher degree
of regularisation cannot be achieved with the proposed method.
\end{abstract}

\section{Introduction}

A copula is a mathematical tool designed to join marginal probability
distribtions forming a joint multidimensional probability distribution. The
advantage with copulas is that they reduce the problem of studying
multidimensional probability distributions to studying a function, copula,
defined on the unit hypercube. For a full introduction to copulas we refer
to Nelsen \cite{Nelsen}. As pointed out in e.g. \cite{GenestEtAl},\cite{Yan},%
\cite{Patton} copulas have recently gained a lot in popularity. This is in
part driven by their usefulness in applied mathematics, especially
mathematical finance. This paper as well is motivated by a problem
originating from applied mathematics but the application is to be found in
toxicology.

Let $\mathbb{I}=\left[ 0,1\right] $ be the unit interval, and let $\mathbb{I}%
^{2}=\mathbb{I}\times \mathbb{I}$. A function $C\left( u,v\right) $ with
piecewise continuous second derivatives is a cumulative distribution
function on $\mathbb{I}^{2}$ if and only if 
\begin{eqnarray}
C_{uv}^{\prime \prime }\left( u,v\right) &\geq &0\text{, }u,v\in \left(
0,1\right) \text{, } \\
C\left( u,0\right) &=&0\text{, }u\in \left[ 0,1\right] \text{, and} \\
C\left( 0,v\right) &=&0\text{, }v\in \left[ 0,1\right]
\end{eqnarray}%
The lower boundary values $C\left( u,0\right) ,C\left( 0,v\right) $ are
zero; the upper boundary values $C\left( u,1\right) ,C\left( 1,v\right) $
are the marginal distributions. The function $C$ is said to be a \emph{copula%
} if the marginals are uniform which is summarised in the following
definition.

\begin{definition}
\label{Def:Copula}A two-dimensional \emph{copula} is a function $C$ from $%
\mathbb{I}^{2}$ to $\mathbb{I}$ with the following properties:\newline
1. For every $u,v$ in $\mathbb{I}$,%
\begin{equation}
C(u,0)=0=C(0,v)
\end{equation}%
and%
\begin{equation}
C(u,1)=u\text{ \ \ and \ \ }C(1,v)=v;
\end{equation}%
2. For every $u_{1},u_{2},v_{1},v_{2}$ in $\mathbb{I}$ such that $u_{1}\leq
u_{2}$ and $v_{1}\leq v_{2},$%
\begin{equation}
C(u_{2},v_{2})-C(u_{2},v_{1})-C(u_{1},v_{2})+C(u_{1},v_{1})\geq 0.
\end{equation}
\end{definition}

Let us consider an example.

\begin{example}
The function $C(u,v)=uv$ is a copula since $C(u,1)=u$, $C(1,v)=v$, $%
C(u,0)=0=C(0,v)$ and%
\begin{equation*}
C(u_{2},v_{2})-C(u_{2},v_{1})-C(u_{1},v_{2})+C(u_{1},v_{1})=(u_{2}-u_{1})(v_{2}-v_{1})\geq 0%
\text{.}
\end{equation*}
\end{example}

It turns out that any copula $C(u,v)$ is bounded from above by the Frech\'{e}%
t-Hoeffding upper bound, and from below by the Frech\'{e}t-Hoeffding lower
bound. Indeed the Frech\'{e}t-Hoeffding bounds are extremal copulas as they
are copulas themselves, which is stated in the theorem below.

\begin{definition}
The Frech\'{e}t-Hoeffding upper bound is given by $M=\min (u,v)$ and the
Frech\'{e}t-Hoeffding lower bound is given by $W=\max (u+v-1)$.
\end{definition}

\begin{theorem}
The upper and lower Frech\'{e}t-Hoeffding bounds $M$ and $W$ are copulas.
Moreover, for any copula $C$ and all $(u,v)\in \mathbb{I}$,%
\begin{equation}
W(u,v)\leq C(u,v)\leq M(u,v).
\end{equation}
\end{theorem}

While generalising a model, the probit model \cite{Probit}, for estimating
the health impact on people exposed to hazardous chemicals we realised that
the problem can be cast as a copula problem. The standard probit model is
used to find a statistical description of the injury outcome of a population
exposed to toxic substances. The model can be utilized to describe this
outcome at several discrete injury states, such as \emph{light injuries}, 
\emph{severe injuries}, and \emph{death}. The joint probability distribution
of two succeeding injury states may be expressed as a multivariate normal
distribution with normal marginal distributions. By Sklar's theorem, see
e.g. \cite{Nelsen}, there exists a copula that prescribes how to joint,
couple, the marginal distributions to achieve the probit model. As it
happens the corresponding copula is the upper Frech\'{e}t-Hoeffding copula $%
M $. Our generalisation of the probit model aims at treating the population
as an aggregate of individuals rather than as one entity. In this approach
individuals are allocated one threshold value per injury state which
determines at which exposure they will reach every state. The thresholds are
distributed among the population according to the standard normal
distribution. However, individuals must be able to possess different
sensibility for the different injury states which means that they will not
be located on the diagonal axis on the multivariate probability distribution
which otherwise would be the trivial solution. As the thresholds are
distributed, the overall statistics of the population must be kept intact.
In copula terms this equates to rearranging the mass of the joint
probability distribution without changing the marginal distributions..
Adding the criteria that the probit model should readily support numerical
investigation we arrived at the condition that the generalised probit model
is feasible if the upper Frech\'{e}t-Hoeffding copula $M$ can be regularized
and still remain a copula. In this paper we will focus on this mathematical
problem while the full toxicological motivation and the generalised probit
model will be published elsewhere in the toxicology literature.\newline
The regularity that we are aiming for is absolute continuity, which is
defined in the following way:

\begin{definition}
A two-dimensional copula $C$ is said to be \emph{absolutely continuous} if $%
C(u,v)$ has a density $c(u,v)$ and%
\begin{equation}
C(u,v)\equiv \int_{0}^{u}\int_{0}^{v}c(s,t)dtds\text{ \ for all }(u,v)\in 
\mathbb{I}^{2}.
\end{equation}
\end{definition}

We note that $\partial ^{2}C/\partial u\partial v$ exists almost everywhere
in $\mathbb{I}^{2}$, see Theorem 2.2.7 in \cite{Nelsen}, and is equal to the
density $c(u,v)$. The example copula $C(u,v)=uv$ that we considered earlier
is absolutely continuous.

\begin{example}
(cont.) The copula $C(u,v)=uv$ is absolutely continuous since%
\begin{equation*}
\int_{0}^{u}\int_{0}^{v}1dsdt=uv\text{ for all }(u,v)\in \mathbb{I}^{2}.%
\text{ }
\end{equation*}
\end{example}

The Frech\'{e}t-Hoeffding copulas on the other hand are not absolutely
continuous, indeed they are not absolutely continuous on any subdomain $%
\subset $ $\mathbb{I}^{2}$, and are thus \emph{singular}. Geometrically the
Frech\'{e}t-Hoeffding upper and lower copula are $C^{0}$ surfaces each made
up of two planar surfaces intersecting transversally along a line segment,
see figure \ref{fig:FHcopulas}. Regularising these functions, $W$ and $M$,
is, of course, trivial: any type of averaging around and along the
transversal intersection will smooth out the $C^{1}$-discontinuity. The
point here is that the regularized function should still be a copula. For
our purposes, the probit model, it would suffice to work with only the upper
Frech\'{e}t-Hoeffding $M$ copula, but the regularization method works
equally well for the lower one $W$. Thus the main contribution in this paper
is that we prove under which conditions the upper and lower Frech\'{e}%
t-Hoeffding copula can be regularised to \emph{absolutely continuous copulas}%
, see Theorem \ref{Thm:Main}. We then show that this class of copulas is
non-empty by providing an explicit example. We remark that the proposed
averaging method can not be used to prove regularity beyond, at most, $C^{2}$
due to unbounded third derivatives.

The question of regularity of copulas and the associated differentiability
of copulas has been studied before and there are at least two well-known
classes of differentiable copulas: Archimedean and Farlie-Gumbel-Morgenstern
copulas. Recently a preprint claiming to characterize all twice
differentiable copulas was presented \cite{MJK2012}. In that paper also a
new class, Fourier copulas, of twice differentiable copulas is introduced.

\begin{figure}[!ht]
\centering
\includegraphics{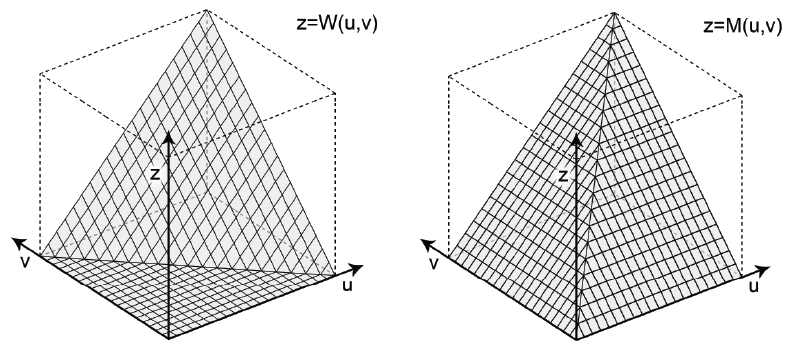}
\caption{Visualisation of the graph of
the Frech\'{e}t-Hoeffding copulas $M$ and $W$.}
\label{fig:FHcopulas}
\end{figure}

\section{Setting of the problem}

As was hinted at in the introduction, regularization of the upper and lower
Frech\'{e}t-Hoeffding copulas $M$ and $W$ will be achieved by averaging
these copulas at, and in the vicinity of, their respective singularities:
for $M$ the singularity is located along the diagonal while for $W$ it is
the anti-diagonal (see figure \ref{fig:FHcopulas}). As averaging method we
choose a family of discs and in each such disc we replace the value of the
copula with the average value of the copula in that disc. Due to the simple
geometry it turns out that these averages can be expressed explicitly,
however the expressions become somewhat involved due to evaluating integrals
on circle segments.

It is convenient to make the following change of coordinates%
\begin{eqnarray}
w &=&\left( v+u-1\right) /\sqrt{2} \\
z &=&\left( v-u\right) /\sqrt{2}
\end{eqnarray}%
and study the problem in $(w,z)$-coordinates. Let $C(u,v)$ be a copula
defined on $\mathbb{I}^{2}$. With a slight abuse of notation, we write $%
C\left( w,z\right) =C\left( u,v\right) $. $C(w,z)\,$is defined on%
\begin{equation}
U=\left\{ \left( w,z\right) \in \mathbb{R}^{2}:\left\vert w\right\vert
+\left\vert z\right\vert \leq 1/\sqrt{2}\right\} \text{.}
\end{equation}%
The copula conditions, stated in Definition \ref{Def:Copula}, can be
summarized in the following way%
\begin{eqnarray}
C_{uv}^{\prime \prime }\left( u,v\right) &\geq &0\text{, }\left( u,v\right)
\in int(\mathbb{I}^{2})  \label{def:Copula1} \\
C &=&\max \left( 0,u+v-1\right) =\min \left( u,v\right) \text{, }\left(
u,v\right) \in \partial \mathbb{I}^{2},  \label{def:Copula2}
\end{eqnarray}%
and may be checked in $(w,z)$-coordinates instead. The inverse coordinate
transformation is given by%
\begin{eqnarray}
u &=&\frac{w-z}{\sqrt{2}}+\frac{1}{2} \\
v &=&\frac{w+z}{\sqrt{2}}+\frac{1}{2}\text{,}
\end{eqnarray}%
which implies that the partial derivatives are transformed according to%
\begin{eqnarray}
\frac{\partial }{\partial u} &=&\frac{1}{\sqrt{2}}\left( \frac{\partial }{%
\partial w}-\frac{\partial }{\partial z}\right) \text{, }\frac{\partial }{%
\partial v}=\frac{1}{\sqrt{2}}\left( \frac{\partial }{\partial w}+\frac{%
\partial }{\partial z}\right) \text{, and} \\
\frac{\partial ^{2}}{\partial u\partial v} &=&\frac{1}{2}\left( \frac{%
\partial ^{2}}{\partial w^{2}}-\frac{\partial ^{2}}{\partial z^{2}}\right) 
\text{.}
\end{eqnarray}%
and thus the copula conditions (\ref{def:Copula1}) and (\ref{def:Copula2})
becomes%
\begin{eqnarray}
C_{ww}^{\prime \prime }-C_{zz}^{\prime \prime } &\geq &0\text{, }\left(
w,z\right) \in int(U)  \label{eqn:defCopulaC_wz''} \\
C &=&\max \left( 0,\sqrt{2}w\right) =\frac{w-\left\vert z\right\vert }{\sqrt{%
2}}+\frac{1}{2}\text{, }\left( w,z\right) \in \partial U
\label{eqn:defCopulaC_wzBoundary}
\end{eqnarray}%
in $(w,z)$-coordinates and where the interior $int(U)$ and the boundary $%
\partial U$ are given in the $\left( w,z\right) $--coordinates by:%
\begin{eqnarray}
int(U) &=&\left\{ \left( w,z\right) \in \mathbb{R}^{2}:\left\vert
w\right\vert +\left\vert z\right\vert <1/\sqrt{2}\right\} \\
\partial U &=&\left\{ w,z\in \mathbb{R}^{2}:\left\vert w\right\vert
+\left\vert z\right\vert =1/\sqrt{2}\right\}
\end{eqnarray}%
The expression for the boundary values occuring in the copula conditions is
defined for all $\left( w,z\right) \in U$:%
\begin{equation}
W=\frac{w+\left\vert w\right\vert }{\sqrt{2}}=\max \left( 0,\sqrt{2}w\right)
=\max \left( 0,v+u-1\right) \text{.}  \label{eqn:W}
\end{equation}%
This is a (non--smooth) copula, the lower Fr\'{e}chet-Hoeffding bound, where
all the probability mass is concentrated on the antidiagonal $\left\{
u+v=1\right\} =\left\{ w=0\right\} $. There is also an upper bound%
\begin{equation}
M=\frac{w-\left\vert z\right\vert }{\sqrt{2}}+\frac{1}{2}=\min \left( \frac{%
w-z}{\sqrt{2}}+\frac{1}{2},\frac{w+z}{\sqrt{2}}+\frac{1}{2}\right) =\min
\left( u,v\right)  \label{eqn:M}
\end{equation}%
which is a non--smooth copula, the upper Fr\'{e}chet-Hoeffding bound, where
all the probability mass is concentrated on the diagonal $\left\{
u=v\right\} =\left\{ z=0\right\} $. Now we are in a position to state the
main theorem describing under which conditions the Frech\'{e}t-Hoeffding
bounds can be regularised using the proposed averaging method.

\begin{theorem}
\label{Thm:Main}Let $r\left( z,w\right) $ be a twice continuously
differentiable and strictly positive function defined for $\left\vert
w\right\vert +\left\vert z\right\vert <1/\sqrt{2}$. Let 
\begin{equation}
w=\left( v+u+1\right) /\sqrt{2},z=\left( v-u\right) /\sqrt{2}\text{,}
\end{equation}%
and let $D_{r}\left( u,v\right) $ be the circular disk with radius $r$
centered at $\left( u,v\right) $, and let%
\begin{equation}
g\left( \rho \right) =\left\{ 
\begin{array}{ccc}
\left. 2\left( \rho \arcsin \left( \rho \right) +\sqrt{1-\rho ^{2}}\left.
\left( 2+\rho ^{2}\right) \right/ 3\right) \right/ \pi & \text{if} & 
\left\vert \rho \right\vert <1 \\ 
\left\vert \rho \right\vert & \text{if} & \left\vert \rho \right\vert \geq 1%
\text{.}%
\end{array}%
\right.
\end{equation}%
Then

\begin{enumerate}
\item If $\left( r_{z}^{\prime }\right) ^{2}\leq \left( 1/2-\left\vert
r_{w}^{\prime }\right\vert \right) ^{2}+3/4$ and $r_{zz}^{\prime \prime
}\leq r_{ww}^{\prime \prime }$ then%
\begin{equation}
\bar{W}\left( u,v\right) \equiv \frac{1}{\pi r^{2}\left( w,z\right) }%
\diint\limits_{D_{r\left( w,z\right) }\left( u,v\right) }W\left( u^{\prime
},v^{\prime }\right) du^{\prime }dv^{\prime }  \label{eqn:defW_bar}
\end{equation}%
is an absolutely continuous copula with probability mass supported on $%
\left\{ \left\vert w\right\vert \leq r\left( w,z\right) \right\} $, and%
\begin{equation}
\bar{W}\left( u,v\right) =\frac{1}{\sqrt{2}}\left( w+r\left( w,z\right)
g\left( \frac{w}{r\left( w,z\right) }\right) \right)  \label{eqn:W_bar}
\end{equation}

\item If $\left( r_{w}^{\prime }\right) ^{2}\leq \left( 1/2-\left\vert
r_{z}^{\prime }\right\vert \right) ^{2}+3/4$ and $r_{ww}^{\prime \prime
}\leq r_{zz}^{\prime \prime }$ then%
\begin{equation}
\bar{M}\left( u,v\right) \equiv \frac{1}{\pi r^{2}\left( w,z\right) }%
\diint\limits_{D_{r\left( w,z\right) }\left( u,v\right) }M\left( u^{\prime
},v^{\prime }\right) du^{\prime }dv^{\prime }  \label{eqn:defM_bar}
\end{equation}%
is an absolutely continuous copula with probability mass supported on $%
\left\{ \left\vert z\right\vert \leq r\left( w,z\right) \right\} $ and%
\begin{equation}
\bar{M}\left( u,v\right) =\frac{1}{\sqrt{2}}\left( w+\sqrt{2}-r\left(
w\right) g\left( \frac{z}{r\left( w,z\right) }\right) \right)
\label{eqn:M_bar}
\end{equation}
\end{enumerate}
\end{theorem}

We remark that the conditions stated in the theorem are not sharp, and this
is pointed out in the proof which we postpone to the next section. Let us
first consider an example showing that the class of copulas generated by the
theorem is non-empty.

\begin{example}
Symmetric Gaussian copula. Let $\Phi $ be the cumulative distribution
function of a standard normal random variable, $\varphi \left( x\right)
=\Phi ^{\prime }\left( x\right) =\exp \left( -x^{2}/2\right) /\sqrt{2\pi 
\text{ }}$its probability density function, and consider a radius function $%
r\left( w\right) $ implicitly defined by%
\begin{equation}
\Phi ^{-1}\left( \frac{w+r\left( w\right) }{\sqrt{2}}+\frac{1}{2}\right)
-\Phi ^{-1}\left( \frac{w-r\left( w\right) }{\sqrt{2}}+\frac{1}{2}\right) =d
\end{equation}%
where $d>0$ is a given constant. Denote%
\begin{eqnarray}
u &=&\frac{w-r\left( w\right) }{\sqrt{2}}+\frac{1}{2}\text{, }x=\Phi
^{-1}\left( u\right) \text{, } \\
v &=&\frac{w+r\left( w\right) }{\sqrt{2}}+\frac{1}{2}\text{, }y=\Phi
^{-1}\left( v\right) \text{.}
\end{eqnarray}%
Implicit differentiation gives 
\begin{equation}
\frac{1}{\varphi \left( y\right) }\frac{1+r^{\prime }}{\sqrt{2}}-\frac{1}{%
\varphi \left( x\right) }\frac{1-r^{\prime }}{\sqrt{2}}=0
\end{equation}%
which gives%
\begin{equation}
r^{\prime }=\frac{\varphi \left( y\right) -\varphi \left( x\right) }{\varphi
\left( y\right) +\varphi \left( x\right) }=\tanh \left( s\right) \text{,}
\end{equation}%
where 
\begin{equation}
s=\left( x^{2}-y^{2}\right) /4=-d\left( x+y\right) /4
\end{equation}%
by the definitions above. Hence $\left\vert r^{\prime }\left( w\right)
\right\vert <1$. Furthermore, 
\begin{equation}
r^{\prime \prime }=\tanh ^{\prime }\left( s\right) s^{\prime }=-d\left(
1-\tanh ^{2}\left( s\right) \right) \left( x^{\prime }+y^{\prime }\right) /4
\end{equation}%
hence it suffices to show that $x^{\prime }\left( w\right) +y^{\prime
}\left( w\right) \geq 0$ for $r_{ww}^{\prime \prime }\leq r_{zz}^{\prime
\prime }=0$. But 
\begin{equation}
x^{\prime }+y^{\prime }=\frac{1}{\varphi \left( x\right) }\frac{1-r^{\prime }%
}{\sqrt{2}}+\frac{1}{\varphi \left( y\right) }\frac{1+r^{\prime }}{\sqrt{2}}%
\geq 0
\end{equation}%
since $\left\vert r^{\prime }\left( w\right) \right\vert <1$ and $\varphi >0$%
. Hence $r\left( w\right) $ satisfies the assumptions in the theorem, so $%
\bar{M}\left( u,v\right) $ is an absolutely continuous copula with
probability mass supported on $\left\vert z\right\vert \leq r\left( w\right) 
$.\newline
It may be interesting to note that upon transforming the copula random
variables $\left( U,V\right) $ back to the original $X=\Phi ^{-1}\left(
U\right) $, $Y=\Phi ^{-1}\left( V\right) $ we obtain a pair of Gaussian
random variables $\left( X,Y\right) $ with standard normal marginal
distributions $p_{X}\left( x\right) =\varphi \left( x\right) $, $p_{Y}\left(
y\right) =\varphi \left( y\right) $ and joint pdf $p_{X,Y}\left( x,y\right) $
supported on $\left\vert y-x\right\vert \leq d$.
\end{example}

As the example shows, verifying the conditions in Theorem \ref{Thm:Main} may
require some work. In the example the radius function $r(w,z)$ did not have
a general dependence on both variables, indeed $r(w,z)=r(w)$ in the example,
and this simplified the verification of the conditions. In general the
conditions are easier to verify if the radius function has a
product-structure $r(w,z)=p(w)q(z)$ which we state in the following
corollary of Theorem \ref{Thm:Main}.

\begin{corollary}
If $r\left( w,z\right) =p\left( w\right) q\left( z\right) $ and $p,q>0$ then
the conclusions for $\bar{M}$ are true if%
\begin{eqnarray}
q\left( z\right) ^{2}\left( p^{\prime }\left( w\right) \right) ^{2} &\leq
&\left( 1/2-\left\vert q^{\prime }\left( z\right) \right\vert p\left(
w\right) \right) ^{2}+3/4 \\
\frac{p^{\prime \prime }\left( w\right) }{p\left( w\right) } &\leq &\frac{%
q^{\prime \prime }\left( z\right) }{q\left( z\right) }
\end{eqnarray}%
for all $\left\vert w\right\vert +\left\vert z\right\vert <1/\sqrt{2}$. In
particular, it holds true if $\varepsilon \in \left[ -1,1\right] $, $p\left(
w\right) >0$ and $q\left( z\right) =\left( 1+\sqrt{2}\varepsilon z\right) $
and%
\begin{eqnarray}
\left( p^{\prime }\left( w\right) \right) ^{2} &\leq &\frac{\left( 1/2-\sqrt{%
2}\left\vert \varepsilon \right\vert p\left( w\right) \right) ^{2}+3/4}{%
\left( 1+\sqrt{2}\varepsilon z\right) ^{2}} \\
p^{\prime \prime }\left( w\right) &\leq &0
\end{eqnarray}%
for all $\left\vert w\right\vert +\left\vert z\right\vert <1/\sqrt{2}$. In
the symmetric case $\varepsilon =0$, the condition reduces to%
\begin{eqnarray}
\left( p^{\prime }\left( w\right) \right) ^{2} &\leq &1 \\
p^{\prime \prime }\left( w\right) &\leq &0\text{.}
\end{eqnarray}
\end{corollary}

\begin{remark}
This allows for skew distributions, since the probability mass of $\bar{M}$
is supported on $\left\vert z\right\vert \leq \left( 1+\sqrt{2}\varepsilon
z\right) p\left( w\right) $, i.e., 
\begin{equation}
-\frac{p\left( w\right) }{1+\sqrt{2}\varepsilon p\left( w\right) }\leq z\leq 
\frac{p\left( w\right) }{1-\sqrt{2}\varepsilon p\left( w\right) }\text{.}
\end{equation}%
The ratio $\kappa $ between modulus of the upper and lower limit is%
\begin{equation}
\kappa =\frac{1+\sqrt{2}\varepsilon p\left( w\right) }{1-\sqrt{2}\varepsilon
p\left( w\right) }\leq \frac{1+\sqrt{2}p\left( w\right) }{1-\sqrt{2}p\left(
w\right) }
\end{equation}
\end{remark}

The corollary and the remark is stated for $\bar{M}(u,v)$, but the analogous
result holds for $\bar{W}\left( u,v\right) $ by interchanging the roles of $%
p(w)$ and $q(z)$, and $w$ and $z$.

\section{Proof of Theorem \protect\ref{Thm:Main}}

The proof of Theorem \ref{Thm:Main} relies on being able to construct radius
functions $r=r\left( w,z\right) $ such that the circular disc averages of
the functions $\left\vert w\right\vert $ and $\left\vert z\right\vert $,%
\begin{eqnarray}
F\left( w,z\right) &=&\frac{1}{\pi r^{2}}\diint\limits_{D_{r}\left(
w,z\right) }\left\vert w^{\prime }\right\vert dw^{\prime }dz^{\prime } \\
G\left( w,z\right) &=&\frac{1}{\pi r^{2}}\diint\limits_{D_{r}\left(
w,z\right) }\left\vert z^{\prime }\right\vert dw^{\prime }dz^{\prime },
\end{eqnarray}%
satisfy%
\begin{eqnarray}
F_{ww}^{\prime \prime }-F_{zz}^{\prime \prime } &\geq &0 \\
G_{zz}^{\prime \prime }-G_{ww}^{\prime \prime } &\geq &0\text{, }
\end{eqnarray}%
where $D_{r}(w,z)$ is the circular disc with radius $r$ centred at $(w,z)$.
Once we have established what conditions the radius functions $r=r(w,z)$
have to satisfy (the averaging Lemma \ref{Lemma:Average}\ below) we pick
those radius functions and apply the disc averaging method to the upper and
the lower Frech\'{e}t-Hoeffding copulas $M$ and $W$ and show that it results
in absolutely continuous copulas: regularised Frech\'{e}t-Hoeffding bounds $%
\overline{M}$ and $\overline{W}$.

Evaluating the disc averages of the functions $\left\vert w\right\vert $ and 
$\left\vert z\right\vert $ and exploring the differentiability properties of
these averages with respect to the choice of radius function $r(w,z)$ and
the part of the domain that is considered is summarised in the following
technical averaging Lemma.

\begin{lemma}
\label{Lemma:Average}Let $r\left( w,z\right) $ be twice continuously
differentiable and strictly positive, and let $D_{r}(w,z)$ be the circular
disc with radius $r$ centred at $(w,z)$, and let%
\begin{equation}
g\left( \rho \right) =\left\{ 
\begin{array}{ccc}
\left. 2\left( \rho \arcsin \left( \rho \right) +\sqrt{1-\rho ^{2}}\left.
\left( 2+\rho ^{2}\right) \right/ 3\right) \right/ \pi & \text{if} & 
\left\vert \rho \right\vert <1 \\ 
\left\vert \rho \right\vert & \text{if} & \left\vert \rho \right\vert \geq 1%
\text{.}%
\end{array}%
\right.  \label{eqn:g}
\end{equation}%
Then 
\begin{eqnarray}
F\left( w,z\right) &\equiv &\frac{1}{\pi r\left( w,z\right) ^{2}}%
\diint\limits_{D_{r\left( w,z\right) }\left( w,z\right) }\left\vert
w^{\prime }\right\vert du^{\prime }dv^{\prime }=r\left( w,z\right) g\left( 
\frac{w}{r\left( w,z\right) }\right)  \label{eqn:F} \\
G\left( w,z\right) &\equiv &\frac{1}{\pi r\left( w,z\right) ^{2}}%
\diint\limits_{D_{r\left( w,z\right) }\left( w,z\right) }\left\vert
z^{\prime }\right\vert dw^{\prime }dz^{\prime }=r\left( w,z\right) g\left( 
\frac{z}{r\left( w,z\right) }\right)  \label{eqn:G}
\end{eqnarray}%
and $F$, $G$ are twice continuously differentiable. Moreover,

\begin{enumerate}
\item If $\left\vert w\right\vert \geq r\left( w,z\right) $ then $F\left(
w,z\right) =\left\vert w\right\vert $ and $F_{ww}^{\prime \prime
}-F_{zz}^{\prime \prime }=0$,

\item if $\left\vert z\right\vert \geq r\left( w,z\right) $ then $G\left(
w,z\right) =\left\vert z\right\vert $ and $G_{zz}^{\prime \prime
}-G_{ww}^{\prime \prime }=0$,

\item if $\left\vert w\right\vert <r\left( w,z\right) $, $\left(
r_{z}^{\prime }\right) ^{2}\leq \left( 1/2-\left\vert r_{w}^{\prime
}\right\vert \right) ^{2}+3/4$ and $r_{zz}^{\prime \prime }\leq
r_{ww}^{\prime \prime }$, then $F_{ww}^{\prime \prime }-F_{zz}^{\prime
\prime }\geq 0$, and

\item if $\left\vert z\right\vert <r\left( w,z\right) $, $\left(
r_{w}^{\prime }\right) ^{2}\leq \left( 1/2-\left\vert r_{z}^{\prime
}\right\vert \right) ^{2}+3/4$ and $r_{ww}^{\prime \prime }\leq
r_{zz}^{\prime \prime }$, then $G_{zz}^{\prime \prime }-G_{ww}^{\prime
\prime }\geq 0$.\qquad
\end{enumerate}
\end{lemma}

\begin{proof}
Let us begin by showing the equality in equations (\ref{eqn:F}) and (\ref%
{eqn:G}), that is, evaluating the disc averages. Let $\rho (w,z)=z/r(w,z).$
The case $\left\vert \rho \right\vert \geq 1$ is immediate, since then the
integrand is linear. Thereby statement 1. and 2. follow directly.\newline
For the case $\left\vert \rho \right\vert <1$ we introduce the
nondimensional coordinates 
\begin{eqnarray}
\omega &=&\frac{w^{\prime }-w}{r}\text{, }\zeta =\frac{z^{\prime }-z}{r} \\
w^{\prime } &=&w+r\omega \text{, }z^{\prime }=z+r\zeta \\
dw^{\prime } &=&rd\omega \text{, }dz^{\prime }=rd\zeta
\end{eqnarray}%
whereby%
\begin{multline*}
G\left( w,z\right) =\frac{1}{\pi r^{2}}\diint\limits_{D_{r}\left( w,z\right)
}\left\vert z^{\prime }\right\vert dw^{\prime }dz^{\prime }=\frac{r}{\pi }%
\diint\limits_{\omega ^{2}+\zeta ^{2}\leq 1}\left\vert \frac{z}{r}+\zeta
\right\vert d\zeta d\omega \\
=-\frac{r}{\pi }\int_{-1}^{-z/r}\left( \frac{z}{r}+\zeta \right) 2\sqrt{%
1-\zeta ^{2}}d\zeta +\frac{r}{\pi }\int_{-z/r}^{1}\left( \frac{z}{r}+\zeta
\right) 2\sqrt{1-\zeta ^{2}}d\zeta \\
=\frac{2z}{\pi }\int_{-z/r}^{1}\sqrt{1-\zeta ^{2}}d\zeta -\frac{2z}{\pi }%
\int_{-1}^{-z/r}\sqrt{1-\zeta ^{2}}d\zeta \\
+\frac{r}{\pi }\int_{-z/r}^{1}2\zeta \sqrt{1-\zeta ^{2}}d\zeta -\frac{r}{\pi 
}\int_{-1}^{-z/r}2\zeta \sqrt{1-\zeta ^{2}}d\zeta .
\end{multline*}%
Applying the standard integrals%
\begin{eqnarray}
\int 2\zeta \sqrt{1-\zeta ^{2}}d\zeta &=&-\frac{2}{3}\left( 1-\zeta
^{2}\right) ^{3/2}+C \\
\int \sqrt{1-\zeta ^{2}}d\zeta &=&\frac{1}{2}\arcsin \left( \zeta \right) +%
\frac{1}{2}\zeta \sqrt{1-\zeta ^{2}}
\end{eqnarray}%
gives us%
\begin{multline*}
G\left( w,z\right) =\frac{1}{\pi r^{2}}\diint\limits_{D_{r}\left( w,z\right)
}\left\vert z^{\prime }\right\vert dw^{\prime }dz^{\prime } \\
=\frac{z}{\pi }\left[ \arcsin \left( \zeta \right) +\zeta \sqrt{1-\zeta ^{2}}%
\right] _{-z/r}^{1}-\frac{z}{\pi }\left[ \arcsin \left( \zeta \right) +\zeta 
\sqrt{1-\zeta ^{2}}\right] _{-1}^{-z/r} \\
+\frac{r}{\pi }\left[ -\frac{2}{3}\left( 1-\zeta ^{2}\right) ^{3/2}\right]
_{-z/r}^{1}-\frac{r}{\pi }\left[ -\frac{2}{3}\left( 1-\zeta ^{2}\right)
^{3/2}\right] _{-1}^{-z/r} \\
=\frac{2z}{\pi }\arcsin \left( \frac{z}{r}\right) +\frac{2z}{\pi }\frac{z}{r}%
\sqrt{1-\frac{z^{2}}{r^{2}}}+\frac{4r}{3\pi }\left( 1-\frac{z^{2}}{r^{2}}%
\right) ^{3/2} \\
=r\cdot \frac{2}{\pi }\left( \frac{z}{r}\arcsin \left( \frac{z}{r}\right) +%
\sqrt{1-\frac{z^{2}}{r^{2}}}\left( \frac{z^{2}}{r^{2}}+\frac{2}{3}\left( 1-%
\frac{z^{2}}{r^{2}}\right) \right) \right) \\
=r\cdot \frac{2}{\pi }\left( \frac{z}{r}\arcsin \left( \frac{z}{r}\right) +%
\frac{1}{3}\sqrt{1-\frac{z^{2}}{r^{2}}}\left( 2+\frac{z^{2}}{r^{2}}\right)
\right) =rg\left( \frac{z}{r}\right) \text{.}
\end{multline*}%
which proves the equality in (\ref{eqn:G}). Proving the disc average $F(w,z)$%
, (\ref{eqn:F}), is analogous.\newline
In a preamble to studying how the derivatives of $F(w,z)$ and $G(w,z)$
behaves, we note that $g\left( \rho \right) \geq 0$ and 
\begin{equation}
g^{\prime }\left( \rho \right) =\left\{ 
\begin{array}{ccc}
\left. 2\left( \arcsin \rho +\rho \sqrt{1-\rho ^{2}}\right) \right/ \pi & 
\text{if} & \left\vert \rho \right\vert <1 \\ 
\rho \left/ \left\vert \rho \right\vert \right. & \text{if} & \left\vert
\rho \right\vert >1%
\end{array}%
\right.  \label{eqn:g'}
\end{equation}%
and%
\begin{equation}
0\leq g^{\prime \prime }\left( \rho \right) =\left\{ 
\begin{array}{ccc}
\left. 4\sqrt{1-\rho ^{2}}\right/ \pi & \text{if} & \left\vert \rho
\right\vert <1 \\ 
0 & \text{if} & \left\vert \rho \right\vert >1%
\end{array}%
\right.  \label{eqn:g''}
\end{equation}%
and $g\left( \rho \right) \rightarrow 1$, $g^{\prime }\left( \rho \right)
\rightarrow \pm 1$ and $g^{\prime \prime }\left( \rho \right) \rightarrow 0$
as $\rho \rightarrow \pm 1$, hence $g\in C^{2}\left( \mathbb{R}\right) $ and
convex. However, $g$ is not $C^{3}$ since%
\begin{equation}
g^{\prime \prime \prime }\left( \rho \right) =-\frac{4}{\pi }\rho \left(
1-\rho ^{2}\right) ^{-1/2}\rightarrow \pm \infty \text{ as }\rho \rightarrow
\pm 1\text{, }\left\vert \rho \right\vert <1\text{. }  \label{eqn:g'''}
\end{equation}%
To simplify calculations, we define the auxiliary function%
\begin{equation}
h\left( \rho \right) =\left\{ 
\begin{array}{ccc}
\left. 4\left( 1-\rho ^{2}\right) ^{\frac{3}{2}}\right/ \left( 3\pi \right)
& \text{if} & \left\vert \rho \right\vert <1 \\ 
0 & \text{if} & \left\vert \rho \right\vert \geq 1%
\end{array}%
\right. ,
\end{equation}%
and note that%
\begin{eqnarray}
h\left( \rho \right) &=&g\left( \rho \right) -\rho g^{\prime }\left( \rho
\right) =\left( 1-\rho ^{2}\right) g^{\prime \prime }\left( \rho \right) /3%
\text{, }\rho \neq 1 \\
h^{\prime }\left( \rho \right) &=&-\rho g^{\prime \prime }\left( \rho
\right) \text{, }\rho \neq 1.
\end{eqnarray}%
And, in addition, we observe that since%
\begin{equation}
\rho \left( w,z\right) =\frac{z}{r\left( w,z\right) }
\end{equation}%
its first partial derivatives are given by%
\begin{equation}
\rho _{w}^{\prime }=-\frac{\rho r_{w}^{\prime }}{r}\text{, }\rho
_{z}^{\prime }=\frac{1-\rho r_{z}^{\prime }}{r}.
\end{equation}%
Now we are ready to study the properties of the second derivatives of $G$,
with the aim of proving points 2. and 4. in the Lemma. Let us remind
ourselves that $G$ is defined as, compare equation (\ref{eqn:G}),%
\begin{equation*}
G(w,z)=r(w,z)g(\rho (w,z))\text{.}
\end{equation*}%
Taking derivatives, first once and then twice, with respect to $w$ and $z$
respectively gives%
\begin{equation}
G_{w}^{\prime }=r_{w}^{\prime }g+rg^{\prime }\rho _{w}^{\prime
}=r_{w}^{\prime }\left( g-\rho g^{\prime }\right) =hr_{w}^{\prime }
\end{equation}%
and%
\begin{equation}
G_{z}^{\prime }=r_{z}^{\prime }g+rg^{\prime }\rho _{z}^{\prime
}=r_{z}^{\prime }\left( g-\rho g^{\prime }\right) +g^{\prime
}=hr_{z}^{\prime }+g^{\prime }
\end{equation}%
and%
\begin{equation}
G_{ww}^{\prime \prime }=h^{\prime }\rho _{w}^{\prime }r_{w}^{\prime
}+hr_{ww}^{\prime \prime }=-\rho g^{\prime \prime }\left( -\rho
r_{w}^{\prime }/r\right) r_{w}^{\prime }+hr_{ww}^{\prime \prime }=g^{\prime
\prime }\left( \rho r_{w}^{\prime }\right) ^{2}/r+hr_{ww}^{\prime \prime }
\label{eqn:G''ww}
\end{equation}%
and%
\begin{equation}
G_{zz}^{\prime \prime }=g^{\prime \prime }\left( \rho r_{z}^{\prime }\right)
^{2}/r+hr_{zz}^{\prime \prime }+g^{\prime \prime }\rho _{z}^{\prime
}=g^{\prime \prime }\rho r_{z}^{\prime }\left( \rho r_{z}^{\prime }-1\right)
/r+hr_{zz}^{\prime \prime }+g^{\prime \prime }/r\text{.}  \label{eqn:G_zz}
\end{equation}%
Rearranging the terms and combining (\ref{eqn:G''ww}) and (\ref{eqn:G_zz})
gives%
\begin{multline*}
r\left( G_{zz}^{\prime \prime }-G_{ww}^{\prime \prime }\right) =g^{\prime
\prime }\left( \left( \rho r_{z}^{\prime }\right) ^{2}-\left( \rho
r_{w}^{\prime }\right) ^{2}+1-\rho r_{z}^{\prime }\right) +h\left(
rr_{zz}^{\prime \prime }-rr_{ww}^{\prime \prime }\right) \\
=g^{\prime \prime }\left( \left( \rho r_{z}^{\prime }\right) ^{2}-\left(
\rho r_{w}^{\prime }\right) ^{2}+1-\rho r_{z}^{\prime }+\left( 1-\rho
^{2}\right) \left( rr_{zz}^{\prime \prime }-rr_{ww}^{\prime \prime }\right)
/3\right) .
\end{multline*}%
Dividing by $g^{\prime \prime }$ we denote the resulting quadratic
polynomial in $\rho $ by $p(\rho )$%
\begin{equation}
\frac{r\left( G_{zz}^{\prime \prime }-G_{ww}^{\prime \prime }\right) }{%
g^{\prime \prime }}=\rho ^{2}\left( \left( r_{z}^{\prime }\right)
^{2}-\left( r_{w}^{\prime }\right) ^{2}-\left( rr_{zz}^{\prime \prime
}-rr_{ww}^{\prime \prime }\right) /3\right) -\rho r_{z}^{\prime }+1+\left(
rr_{zz}^{\prime \prime }-rr_{ww}^{\prime \prime }\right) /3\equiv p\left(
\rho \right) \text{.}  \label{eqn:quadpoly}
\end{equation}%
A quadratic polynomial $p\left( \rho \right) =a\rho ^{2}+b\rho +c$ is $\geq
0 $ for $-1\leq \rho \leq 1$ if and only if 1) The values at the endpoints
are nonnegative, i.e., $p\left( -1\right) \geq 0$ and $p\left( 1\right) \geq
0$, which is equivalent to $\left\vert b\right\vert \leq a+c$, and 2) if
there is local minimum between the endpoints, i.e., if $\left\vert
b\right\vert \leq 2\left\vert a\right\vert $, then this minimum is
nonnegative, i.e, $c\geq b^{2}/\left( 4a\right) $. If $\left\vert
b\right\vert \leq 2\left\vert a\right\vert $ then $b^{2}/\left( 4a\right)
\leq 1$ so a simpler, sufficient condition would be $\left\vert b\right\vert
\leq a+c$ and $c\geq 1$. Applying this condition to (\ref{eqn:quadpoly}) we
see that%
\begin{equation}
\frac{r\left( G_{zz}^{\prime \prime }-G_{ww}^{\prime \prime }\right) }{%
g^{\prime \prime }}\geq 0  \label{eqn:differenceG''}
\end{equation}%
if $\left\vert r_{z}^{\prime }\right\vert \leq 1+\left( r_{z}^{\prime
}\right) ^{2}-\left( r_{w}^{\prime }\right) ^{2}$ and $1+\left(
rr_{z}^{\prime \prime }-rr_{ww}^{\prime \prime }\right) /3\geq 1$, which is
equivalent to 
\begin{eqnarray}
\left( r_{w}^{\prime }\right) ^{2} &\leq &\left( 1/2-\left\vert
r_{z}^{\prime }\right\vert \right) ^{2}+3/4\text{ and}  \label{eqn:cond_r'}
\\
r_{ww}^{\prime \prime } &\leq &r_{zz}^{\prime \prime }\text{.}
\label{eqn:cond_r''}
\end{eqnarray}%
The last condition is only needed when%
\begin{equation}
\left\vert r_{z}^{\prime }\right\vert <2\left\vert \left( r_{z}^{\prime
}\right) ^{2}-\left( r_{w}^{\prime }\right) ^{2}-\left( rr_{zz}^{\prime
\prime }-rr_{ww}^{\prime \prime }\right) /3\right\vert \text{. }
\end{equation}%
\newline
The radius function $r$ is always positive, and by (\ref{eqn:g''}) we have $%
g^{\prime \prime }>0$ for $\left\vert \rho \right\vert <1$, hence equation (%
\ref{eqn:differenceG''}) implies that $G_{zz}^{\prime \prime
}-G_{ww}^{\prime \prime }>0$ if conditions (\ref{eqn:cond_r'}) and (\ref%
{eqn:cond_r''}) are satisfied. Remember that $\rho =z/r(w,z)$, hence $%
\left\vert \rho \right\vert <1$ is equivalent to $\left\vert z\right\vert
<r(w,z)$. Thus we have proved statement 4. in the lemma.\newline
We remark that the conditions stated in the lemma (and Theorem \ref{Thm:Main}%
) are not sharp since we chose a simplified condition in the quadratic
polynomial, if more precision is needed we the exact condition%
\begin{equation}
1+\left( rr_{zz}^{\prime \prime }-rr_{ww}^{\prime \prime }\right) /3\geq 
\frac{\left( r_{z}^{\prime }\right) ^{2}}{4\left( \left( r_{z}^{\prime
}\right) ^{2}-\left( r_{w}^{\prime }\right) ^{2}-\left( rr_{zz}^{\prime
\prime }-rr_{ww}^{\prime \prime }\right) /3\right) }
\end{equation}%
which is equivalent to%
\begin{equation}
\frac{\left( r_{z}^{\prime }\right) ^{2}-4\left( \left( r_{z}^{\prime
}\right) ^{2}-\left( r_{w}^{\prime }\right) ^{2}-\left( rr_{zz}^{\prime
\prime }-rr_{ww}^{\prime \prime }\right) /3\right) \left( rr_{zz}^{\prime
\prime }-rr_{ww}^{\prime \prime }\right) /3}{4\left( \left( r_{z}^{\prime
}\right) ^{2}-\left( r_{w}^{\prime }\right) ^{2}-\left( rr_{zz}^{\prime
\prime }-rr_{ww}^{\prime \prime }\right) /3\right) }\leq 1
\end{equation}%
can be employed.\newline
Next, let us consider the remaining unproved statement, statement 3. It
follows analogously by the arguments above by interchanging the roles of $w$
and $z$. That is, let us consider%
\begin{equation}
F\left( w,z\right) =\frac{1}{\pi r^{2}}\diint\limits_{D_{r}\left( w,z\right)
}\left\vert w^{\prime }\right\vert dw^{\prime }dz^{\prime }.
\end{equation}%
Interchanging the roles of $w$ and $z$, and applying the the same arguments
as above to%
\begin{equation}
F\left( w,z\right) =r\left( w,z\right) g\left( \frac{w}{r\left( w,z\right) }%
\right)
\end{equation}%
we obtain%
\begin{equation}
\frac{r\left( F_{ww}^{\prime \prime }-F_{zz}^{\prime \prime }\right) }{%
g^{\prime \prime }}\geq 0
\end{equation}%
if 
\begin{equation}
\left( r_{z}^{\prime }\right) ^{2}\leq \left( 1/2-\left\vert r_{w}^{\prime
}\right\vert \right) ^{2}+3/4
\end{equation}%
and%
\begin{equation}
r_{zz}^{\prime \prime }\leq r_{ww}^{\prime \prime }\text{.}
\end{equation}%
Thus we have proved the lemma.
\end{proof}

We are now in a position to prove Theorem \ref{Thm:Main}.

\begin{proof}
Let $r(w,z)$ be a twice continously differentiable and strictly positive
function such that $\left( r_{z}^{\prime }\right) ^{2}\leq \left(
1/2-\left\vert r_{w}^{\prime }\right\vert \right) ^{2}+3/4$ and $%
r_{zz}^{\prime \prime }\leq r_{ww}^{\prime \prime }$. Considering the
definition of $\bar{W}\left( u,v\right) $ in (\ref{eqn:defW_bar}) and
inserting the expression for the lower Frech\'{e}t-Hoeffding copula $W$
stated in (\ref{eqn:W}) gives 
\begin{equation*}
\bar{W}\left( u,v\right) =\frac{1}{\pi r^{2}\left( w,z\right) }%
\diint\limits_{D_{r\left( w,z\right) }\left( w,z\right) }\frac{w^{\prime
}+\left\vert w^{\prime }\right\vert }{\sqrt{2}}dw^{\prime }dz^{\prime }.
\end{equation*}%
Applying Lemma \ref{Lemma:Average} we obtain%
\begin{equation*}
\bar{W}\left( u,v\right) =\frac{w+F\left( w,z\right) }{\sqrt{2}}=\frac{1}{%
\sqrt{2}}\left( w+r\left( z\right) g\left( \frac{w}{r\left( w,z\right) }%
\right) \right) .
\end{equation*}%
Likewise, letting $r(w,z)$ be a twice continously differentiable and
strictly positive function such that $\left( r_{w}^{\prime }\right) ^{2}\leq
\left( 1/2-\left\vert r_{z}^{\prime }\right\vert \right) ^{2}+3/4$ and $%
r_{ww}^{\prime \prime }\leq r_{zz}^{\prime \prime }$, and inserting the
expression for the upper Frech\'{e}t-Hoeffding copula $M$, stated in (\ref%
{eqn:M}), into the definition of $\bar{M}\left( u,v\right) $ given in (\ref%
{eqn:defM_bar}) yields 
\begin{multline*}
\bar{M}\left( u,v\right) \equiv \frac{1}{\pi r^{2}\left( w,z\right) }%
\diint\limits_{D_{r\left( w,z\right) }\left( w,z\right) }\frac{w^{\prime
}-\left\vert z^{\prime }\right\vert }{\sqrt{2}}+\frac{1}{2}dw^{\prime
}dz^{\prime } \\
=\frac{1}{\sqrt{2}\pi r^{2}\left( w,z\right) }\diint\limits_{D_{r\left(
w,z\right) }\left( w,z\right) }\left( w^{\prime }+\frac{1}{\sqrt{2}}\right)
dw^{\prime }dz^{\prime }-\frac{1}{\sqrt{2}\pi r^{2}\left( w,z\right) }%
\diint\limits_{D_{r\left( w,z\right) }\left( w,z\right) }\left\vert
z^{\prime }\right\vert dw^{\prime }dz^{\prime }
\end{multline*}%
Applying Lemma \ref{Lemma:Average} gives%
\begin{equation*}
\bar{M}\left( u,v\right) =\frac{w+1/\sqrt{2}-G\left( w,z\right) }{\sqrt{2}}=%
\frac{1}{2}+\frac{1}{\sqrt{2}}\left( w-r\left( w,z\right) g\left( \frac{z}{%
r\left( w,z\right) }\right) \right) .
\end{equation*}%
Since $r,g\in C^{2}$ it follows that $\bar{W}$ and $\bar{M}$ are $C^{2}$;
however they are not $C^{3}$ since $g^{\prime \prime \prime }\left( \rho
\right) $ is unbounded as $\rho \nearrow 1$ or $\rho \searrow -1$. The
sufficient conditions on $r\left( w,z\right) $ in in the lemma are
fullfilled in the respective cases. Therefore $F_{ww}^{\prime \prime
}-F_{zz}^{\prime \prime }\geq 0$ and $G_{zz}^{\prime \prime }-G_{ww}^{\prime
\prime }\geq 0$, and hence it follows, see the copula definition (\ref%
{eqn:defCopulaC_wz''}), that $\bar{W}_{uv}^{\prime \prime }=\bar{W}%
_{ww}^{\prime \prime }-\bar{W}_{zz}^{\prime \prime }\geq 0$ and $\bar{M}%
_{uv}^{\prime \prime }=\bar{M}_{ww}^{\prime \prime }-\bar{M}_{zz}^{\prime
\prime }\geq 0$. Furthermore, the averaging method employed does not alter
the boundary values of the copula, hence also the copula condition (\ref%
{eqn:defCopulaC_wzBoundary}) is satisfied since $W$ and $M$ are copulas.
Thus we have proved that $\bar{W}$ and $\bar{M}$ are absolutely continuous
copulas.
\end{proof}

\end{document}